\documentclass[english,11pt]{smfart}

\usepackage{etex}

\usepackage{amsbsy}
\usepackage{amsmath,amsfonts,amssymb,amsthm,mathrsfs,mathtools}

\usepackage{bm}

\usepackage[a4paper,vmargin={3.5cm,3.5cm},hmargin={2.5cm,2.5cm}]{geometry}
\usepackage[font=sf, labelfont={sf,bf}, margin=1cm]{caption}
\usepackage{graphicx}
\usepackage{epsfig}%pour les eps
\usepackage{latexsym}%encore des symboles
\usepackage{xcolor}
\usepackage{ae,aecompl}
\usepackage{soul,framed}
\usepackage{comment}

\usepackage{xcolor}
\usepackage[pdfpagemode=UseNone,bookmarksopen=false,colorlinks=true,urlcolor=blue,citecolor=blue,citebordercolor=blue,linkcolor=blue]{hyperref}
\usepackage{smfhyperref}
\usepackage[capitalize]{cleveref}

\usepackage{pstricks}
\usepackage{enumerate}
\usepackage{tikz,animate,media9}						%%%%%%%%%%
\usepackage{todonotes}
\usepackage{pifont}
\usepackage{bm,marvosym}
\usepackage{algorithm}
\usepackage{algorithmic}

\usepackage{bbm}

\usepackage{tcolorbox}

\usepackage{natbib}  %Nature-like bibliography

\definecolor{aleacolor}{rgb}{0.16,0.59,0.78}

\renewcommand{\cite}{\citet*}

\hypersetup{
	breaklinks,
	colorlinks=true,
	linkcolor=aleacolor,
	urlcolor=aleacolor,
	citecolor=aleacolor}

\newcommand{\ndN}{\mathbb{N}}

\newcommand{\ndR}{\mathbb{R}}

\def\P{\mathbb{P}}

\renewcommand{\Pr}[1]{\mathbb{P}(#1)}

\newcommand{\Prb}[1]{\mathbb{P}\left(#1\right)}

\newcommand{\Ex}[1]{\mathbb{E}[#1]}

\newcommand{\Exb}[1]{\mathbb{E}\left[#1\right]}

\newcommand{\Va}[1]{\mathbb{V}[#1]}

\newcommand{\one}{{\mathbbm{1}}}

\newcommand{\convdis}{\,{\buildrel d \over \longrightarrow}\,}
\newcommand{\convd}{\,{\buildrel d \over \longrightarrow}\,}

\newcommand{\convp}{\,{\buildrel p \over \longrightarrow}\,}

\newcommand{\eqdist}{\,{\buildrel d \over =}\,}

\newcommand{\he}{\mathrm{h}}

\newcommand{\cA}{\mathcal{A}}

\newcommand{\cE}{\mathcal{E}}

\newcommand{\cS}{\mathcal{S}}
\newcommand{\cT}{\mathcal{T}}
\newcommand{\cU}{\mathcal{U}}

\newcommand{\cW}{\mathcal{W}}

\newcommand{\cZ}{\mathcal{Z}}

\newcommand{\mA}{\mathsf{A}}
\newcommand{\mB}{\mathsf{B}}
\newcommand{\mC}{\mathsf{C}}
\newcommand{\mD}{\mathsf{D}}

\newcommand{\mF}{\mathsf{F}}
\newcommand{\mG}{\mathsf{G}}
\newcommand{\mH}{\mathsf{H}}

\newcommand{\mR}{\mathsf{R}}

\newcommand{\mT}{\mathsf{T}}
\newcommand{\mU}{\mathsf{U}}

\newcommand{\mZ}{\mathsf{Z}}

\newcommand{\ind}{\mathrm{ind}}

\newcommand{\ve}{\mathrm{v}}

\newtheorem{theorem}{Theorem}[section]

\newtheorem{corollary}[theorem]{Corollary}
\newtheorem{proposition}[theorem]{Proposition}
\newtheorem{lemma}[theorem]{Lemma}

\newtheorem{definition}[theorem]{Definition}

\numberwithin{equation}{section}

\title{\textbf{Graphon convergence of  random cographs}}
\date{}

\author{Benedikt Stufler}% \thanks{The author is supported by the Swiss National Science Foundation grant number 200020\_172515.}}
\address[Benedikt Stufler]{Institute of Mathematics, University of Zurich}
\email{benedikt.stufler@math.uzh.ch}

\begin{document}

\vspace {-0.5cm}

\begin{abstract}
We study the behaviour of random labelled and unlabelled cographs with $n$ vertices as $n$ tends to infinity. Our main result is a novel probabilistic limit in the space of graphons. 
\end{abstract}

\maketitle

%\tableofcontents

\section{Introduction}

The collection of cographs may be characterized as the smallest class of finite simple graphs that contains the graph with one vertex and is closed under taking complements and forming finite disjoint unions. Further characterizations and structural properties are known, see~\cite[Thm. 13.7]{MR2063679},~\cite[Thm. 2]{MR619603}, and references given therein. In particular, cographs are related to separable permutations, which received recent attention in~\cite{MR3813988} and \cite{2019arXiv190407135B}.  Cographs were also studied from an enumerative viewpoint, see for example~\cite{MR2154567}, and from an algorithmic viewpoint, see for example~\cite{MR2443114}. They have some relevance in computer science, as certain hard computational problems may be solved in polynomial time if the input is restricted to the class of cographs, see~\cite{MR3942334} and references given therein.

The present work studies the typical properties of random cographs with many vertices. The models under consideration are the graph $\mG_n$ drawn uniformly at random among all cographs with vertices labelled from $1$ to $n$, and the graph $\mU_n$ chosen uniformly among all cographs with $n$ unlabelled vertices. Our main result  states convergence in probability of these models interpreted as random elements in the space of graphons. We recall the essentials of this space in Section~\ref{sec:con} and refer the reader to the book by~\cite{MR3012035} and references given therein for details.

\begin{theorem}
	\label{te:main}
	There is a  graphon $W_\mathrm{1/2}$ such that
	\begin{align}
		\mG_n \convp W_\mathrm{1/2}  \qquad \text{and} \qquad \mU_n \convp W_\mathrm{1/2}.
	\end{align}
\end{theorem}
 Here and throughout the rest we use the usual notation $\convp$ and $\convdis$ for convergence in probability and distribution as $n$ tends to infinity.
The limit object is defined in Definition~\ref{def:Wb} below. We expect it to be universal in the sense that it, and tiltet variants $(W_p)_{0 < p < 1}$ of it, arise as limits of a variety of models of dense random graphs. The proof of our main result uses a result  by \cite[Cor. 3.2]{MR2463439} to reduce the task  to distributional convergence of the subgraph induced by any fixed number of randomly selected vertices (or equivalently convergence in probability of induced homomorphism densities). Our main contribution is in the unlabelled case. The bijection between cographs and cotrees relates this to study of heights in trees induced by uniformly selected leaves in a random P\'olya tree $\mA_n$ with $n$ leaves, where each internal vertex has outdegree at least $2$. Using a relation between P\'olya structures and branching processes discovered in~\cite{MR3773800} and the framework of unlabelled enriched trees from~\cite{StEJC2018}, we relate the study of $\mA_n$ to the study of a reducible $2$-type Galton--Watson tree $\mT_n$ conditioned on producing $n$ leaves. We then proceed to give an extension of the skeleton  decomposition of~\cite{MR1207226} to a general setting of reducible $2$-type Galton--Watson trees (also called sesqui-type trees), similar to the one recently given in~\cite{2019arXiv190407135B}. See Theorem~\ref{te:skeldecomp}. As a special case, this translates to an (extended) skeleton decomposition for the P\'olya tree $\mA_n$. This entails for any integer $k \ge 1$ a local limit theorem for the shape and distances in the tree spanned by the root of $\mA_n$ and $k$ uniformly and independently selected leaves. The bijection between cographs and cotrees then allows us to transfer this to convergence of subgraphs induced by uniformly selected vertices in $\mU_n$. We note that the study of various models of unordered unlabelled trees conditioned to have large number of vertices has received growing attention in recent literature, see \cite{MR2829313}, \cite{MR3050512}, \cite{MR3773800}, \cite{doi:10.1002/rsa.20833}. 

\begin{comment}
The skeleton decomposition for general sesqui-type trees in Theorem~\ref{te:skeldecomp} allows us to apply a result of \cite[Thm. 20]{MR1207226}  to prove the following scaling limit:

\begin{theorem}
	\label{te:main2}
	There is a plane embedding of $\mA_n$ such that the corresponding continuous search depth process $(C_n(t), 0 \le t \le 1)$ admits the Brownian excursion $(\bm{e}(t), 0 \le t \le 1)$ of length~$1$ as scaling limit:
	\begin{align}
		\label{eq:brownian}
		(c_0 n^{-1/2} C_n(t), 0 \le t \le 1) \convdis (\bm{e}(t), 0 \le t \le 1).
	\end{align}
	Here $c_0>0$ denotes a constant given in Equation~\eqref{eq:czero}. Consequently, $\mA_n$ converges in the Gromov--Hausdorff sense to the Brownian tree after rescaling edge lengths by $c_0 n^{-1/2}$.
\end{theorem}

Our methods allow us to argue similarly for random P\'olya trees with $n$ leaves that are subject to additional degree restrictions, although additional care has to be taken when the restrictions introduce periodicities.
\end{comment}

\section{Induced subgraph densities and graphons}
\label{sec:con}
\subsection{Graph limits}
\label{eq:gl}

We follow closely the presentation by~\cite{MR2463439}. All graphs considered in the present work are simple. If $G$ is a finite graph and $v_1, \ldots, v_k$ is a sequences of vertices in $G$, we let $G(v_1, \ldots, v_k)$ denote the graph with vertex set $[k]:=\{1, \ldots, k\}$ such that elements $i,j \in [k]$ are adjacent if and only if $v_i$ and $v_j$ are. If $G$ has at least $k$ vertices we define $G[k]' := G	(v_1', \ldots, v_k')$ where the vertices $v_1', \ldots, v_k'$ are selected uniformly at random without replacement. Given a graph $H$ with vertex set $[k]$, the \emph{induced subgraph density} is defined by 
\begin{align}
	\label{eq:tind}
	t_{\ind}(H, G) = \Pr{ G[k]' = H}.
\end{align}
We let $\cU$ denote the countable class of all finite unlabelled graphs. For any unlabelled graph $U \in \cU$ with $k$ vertices we may select an arbitrary  isomorphic graph $H$ with vertex set $[k]$ and define $t_{\ind}(U, G) := t_{\ind}(H, G)$. By symmetry this does not depend on the choice of $H$ and is hence well-defined. Hence we may consider the map
\begin{align}
	\tau_{\ind}: \cU \to [0,1]^\cU, \quad G \mapsto (t_{\ind}(U,G))_{U \in \cU}.
\end{align}
We let $\bar{\cU}$ denote the closure of the image $\tau_\ind(\cU)$ under this mapping, and set $\cU_\infty = \bar{\cU} \setminus \tau_\ind(\cU)$. Note that $[0,1]^\cU$ is a countable product of Polish spaces and hence Polish. This makes the closed subset $\bar{\cU}$ a Polish space as well. The map $t_{\ind}(H, \cdot): \cU \to \bar{\cU}$ has a unique extension $\bar{\cU} \to \bar{\cU}$ that we also denote by $t_{\ind}(H, \cdot)$.

\begin{lemma}[{\cite[Thm. 3.1, Cor 3.2]{MR2463439}}]
	\label{le:mega}
	The following statements are equivalent for any sequence $(\mC_n)_{n \ge 1}$ of random (unlabelled) graphs whose number of vertices $\ve(\mC_n)$ satisfies $\ve(\mC_n) \convp \infty$.
	\begin{enumerate}
		\item $\mC_n$ converges weakly to some random element $\Gamma$ of $\bar{\cU}$.
		\item For any family $(H_i)_{1 \le i \le \ell}$ of finite graphs the vector $(t_{\ind}(H_i, \mC_n))_{1 \le i \le \ell}$ converges weakly.
		\item For any finite graph $H\in\cU$ the density $t_{\ind}(H, \mC_n)$ converges weakly.
		\item For any finite graph $H\in \cU$ the average density $\Ex{t_{\ind}(H, \mC_n)}$ converges.
	\end{enumerate}
	The limits are necessarily given by $(t_{\ind}(H_i, \Gamma))_{1 \le i \le \ell}$, $t_{\ind}(H, \Gamma)$, and $\Ex{t_{\ind}(H, \Gamma)}$.  Moreover, if we require $\Gamma$ in the first condition to be deterministic, then an analogous statement holds with weak convergence replaced by convergence in probability in all conditions.
\end{lemma}

%Note that the fourth condition corresponds to distributional convergence of the graph $\mC_n[k]$.

\subsection{Graphons}

The presentation in this section follows the comprehensive account on graphons in the book by~\cite{MR3012035}.  
A \emph{graphon} (the name comes from the contraction of ``graph-function'') may be defined as a symmetric measurable function
\[
	W: [0,1]^2 \to [0,1].
\]
We call graphons $V$ and $W$ \emph{weakly isomorphic}, if there exist  measure preserving maps $f$ and $g$ fro the unit interval to itself such that $V(f(x), g(y)) = W(f(x), g(y))$ almost everywhere. We let $\widehat{\cW}_{\mathsf{S}}$ denote the collection of graphons viewed up to weak isomorphism. 

Let $G$ be a graph with vertices $v_1, \ldots, v_n$. We interpret $G$  as a graphon $W_G$ by setting it equal to~$1$ on all squares of the form $](a-1)/n, a/n[ \, \, \times \, \, ](b-1]/n, b/n[$ with integers $1 \le a,b \le n$ such that $v_a$ and $v_b$ are adjacent. The function is extended to the unit square by setting it equal to zero everywhere else.

Let $H$ be a graph with vertices $w_1, \ldots, w_k$. 
For integers $1 \le i,j \le k$ we write $i \sim j$ if $w_i$ and $w_j$ are adjacent, and otherwise $i \nsim j$.  The induced subgraph density $t_{\ind}(H,U)$ is defined by sampling $k$ points $X_1, \ldots, X_k$ of the unit square uniformly and independently at random, and setting
\begin{align}
\label{eq:ind}
t_{\ind}(H,U) = \Exb{ \left(\prod_{i \sim j} U(X_i, X_j) \right)  \left(\prod_{i \nsim j}(1- U(X_i,X_j)) \right)}.
\end{align}
For example, if $U \equiv p$ for some constant $p \in [0,1]$ then $t_{\ind}(H,U) = 2^{-\binom{k}{2}}$. Note that interpreting finite graphs as graphons is compatible with the two induced subgraph frequency definitions:
\[
	t_\ind(H,G) = t_{\ind}(H, W_G).
\]

Several equivalent metrics such as the cut metric $\delta_{\square}$ or the sampling metric $\delta_{\text{samp}}$ are in use on the collection $\widehat{\cW}_{\mathsf{S}}$, see \cite{MR3012035} for details. A sequence $V_1, V_2, \ldots$ of points in  $\widehat{\cW}_{\mathsf{S}}$ convergences to a point $V \in \widehat{\cW}_{\mathsf{S}}$ with respect to any of these metrics if and only 
\[
	\lim_{n \to \infty} t_\ind(H,V_n) = t_\ind(H,V)
\]
for all $H \in \cU$. 
\begin{comment}
The \emph{cut distance} $\delta_\square$ between graphons $V$ and $W$ is defined by
\[
\delta_\square(V,W) = \inf_{f,g} \sup_{A,B} \left| \int_{A \times B} \left(V(f(x),f(y)) - U(g(x),g(y)) \right) \mathrm{d}x \mathrm{d}y \right|.
\]
The cut distance induces a metric on the $\widehat{\cW}_{\mathsf{S}}$, see \cite{MR3012035} and references given therein.
\end{comment}
The map
\begin{align}
	\label{eq:homeo}
	\widehat{\cW}_{\mathsf{S}} \to \cU_\infty, \qquad W \mapsto \Gamma_W = (t_\ind(H,W))_{H \in \cU},
\end{align}
is  a homeomorphism, see \cite[Rem. 6.1]{MR2463439}. In Lemma~\ref{le:mega}, the assumption $\ve(\mC_n) \convp \infty$ ensures that the limit $\Gamma$ (if it exists) is almost surely an element of $\cU_\infty$.  Hence:
\begin{corollary}
	\label{co:fu}
	Let $(\mC_n)_{n \ge 1}$ be a sequence of random (unlabelled) graphs satisfying $\ve(\mC_n) \convp \infty$. If $\mC_n \convd \Gamma$ in $\bar{\cU}$, then $W_{\mC_n} \convd W$ for a random graphon $W$ satisfying $\Gamma_W \eqdist \Gamma$.
\end{corollary}
Of course, if $\Gamma$ (equivalently $W$) is deterministic, then convergence in probability holds.

\section{Cographs and cotrees}
\label{sec:bij}

Cographs admit a bijective encoding as cotrees. We recall this fact following the presentation by~\cite{MR619603}.

A \emph{cotree} is a rooted unordered tree where the leaves carry labels and the internal vertices carry signs ($\oplus$ or $\ominus$). Each internal node is required to have at least $2$ children. The signs along any path connecting the root and a leaf are required to alternate.

For any finite  non-empty set $X$ of labels there is a bijection $\phi_X$ between the cotrees with leaves labelled by that set, and the cographs with vertices labelled by $X$. Here a cotree consisting of a single labelled vertex corresponds to a cograph consisting of a single vertex with the same label. The general case is defined recursively. Given a vertex in a rooted tree, we may consider the \emph{fringe subtree} at that vertex, which is the maximal subtree rooted at that vertex. We refer to the fringe subtrees at offspring of the root as $\emph{branches}$.  A cotree whose root carries a $\ominus$-sign corresponds to the graph \emph{union} of the cographs corresponding its branches. A cotree whose root carries a $\oplus$-sign corresponds to the \emph{join} operation of the cographs corresponding to its branches. That is, we form the union of these cographs and add an edge between any pair of vertices corresponding to different branches.

The correspondence between cotrees and cographs is known to be compatible with graph isomorphisms, that is,  unlabelled cographs correspond bijectively to unlabelled cotrees. The following easy observation  expresses how adjacency in the cograph is reflected in the cotree.

\begin{proposition}
	\label{pro:ad}
	Two vertices in a cograph are adjacent if and only if their lowest common ancestor in the corresponding cotree has label $\oplus$.
\end{proposition}

\section{The limit object}
\label{sec:limobj}

Proposition~\ref{pro:ad} motivates the following generalization of the bijection between cographs and cotrees. We define the class of \emph{generalized cotrees} in the same way as the class of cotrees, but without requiring the signs to alternate on paths from the root to leaves.  For any generalized cotree $D$ we let $\psi(D)$ denote the graph whose vertex set is the set of leaves of $D$, such that two points are adjacent if and only if their lowest common ancestor in $D$ has label $\oplus$.

Given $k \ge 1$  a \emph{proper $k$-tree} is a (planted) plane tree that has precisely $k$ leaves, labelled from $1$ to $k$. The root of a proper $k$-tree is required to have outdegree $1$ and all other internal nodes have outdegree $2$. There are $ 2^{k-1} \prod_{i=1}^{k-1}(2i-1)$ such trees, and each has $2k -1$ edges. We let $\mR_k$ denote the uniformly at random selected proper $k$-tree. The tree $\mR_k$ was shown by \cite{MR1207226} to be the limiting distribution of the genealogical structure of $k$ uniformly selected vertices of large random trees that lie in the universality class of the Brownian tree.

Let $0<p<1$ be a constant. We define a random generalized cotree $\mD_k^p$ by snipping away the root and its only adjacent edge away from $\mR_k$ and assigning to each internal vertex a sign ($\oplus$ or $\ominus$) according to an independent coin flip that yields $\oplus$ with probability $p$. We let $\mH_k^p := \psi(\mD_k^p)$ denote the corresponding random graph. For each unlabelled graph $H \in \cU$ we let $H'$ denote a version with labels $1, \ldots, \ve(H)$ and set $q_{H,p} := \Pr{\mH_n^p= H'}$. Note that this does not depend on the choice of~$H'$.

\begin{definition}
	\label{def:Wb}
For each $0<p<1$ we let $W_p$ denote the graphon corresponding to the family $(q_{H,p})_{H \in \cU}$
	under the homeomorphism \eqref{eq:homeo}.
\end{definition}

Lemma~\ref{le:mega} and Corollary~\ref{co:fu} entail that a sequence $(\mC_n)_{n \ge 1}$ of random finite graphs satisfies 
\begin{align}
	\label{eq:definite}
	\mC_n \convp W_{p} \qquad \text{if and only if} \qquad \mC_n[k]' \convd \mH_k^p \text{ for all $k \ge 1$}.
\end{align}
The motivation for defining this family of tiltings of $W_{1/2}$ is that similarly tilted objects called the Brownian separable permutons are known to arise for classes of random permutations, see~\cite{2017arXiv170608333B} and \cite{2019arXiv190307522B}. This motivates the question, whether graphons from the family $(W_p)_{0<p<1}$ arise for other models of dense random graphs as well.

\section{Random labelled cographs}

In this section we prove Theorem~\ref{te:main} for the labelled case. The unlabelled case, which constitutes the main input of the present work, is treated in Section~\ref{sec:unlabelled}. Consider the collection $\cT$ of unordered rooted trees where each internal vertex has outdegree at least $2$.  A cotree with at least $3$ vertices is obtained in a unique way from such a tree by choosing the sign of the root ($\oplus$ or $\ominus$)  and propagating the signs of the remaining vertices according to the parity of their height. We consider the exponential generating series $\cT(z)$ where $z$ marks the number of leaves. It is easy to see that
\begin{align}
	\cT(z) = z + \sum_{k \ge 2} \frac{\cT(z)^k}{k!}.
\end{align}

This characterizes this class  as a special case of so called Schr\"oder-enriched parenthesizations, see~\cite{MR1284403}. By a general principle (that in this special case is also easy to verify directly) given in \cite[Lem. 6.7 and following paragraphs]{2016arXiv161202580S}, we obtain the following sampling procedure:

\begin{lemma}
	\label{le:samplinglemma}
For $n \ge 3$ we may generate the random cograph $\mG_n$ as follows.
\begin{enumerate}
	\item Let $\tau_n$ denote a Galton--Watson tree conditioned on having $n$ leaves with a critical offspring distribution $\eta$ being given by
	\begin{align}
		\label{eq:xi}
		\Ex{z^\eta} = 2 - \frac{1}{\log(2)} + \sum_{k\ge 2} \frac{z^k \log^{k-1}(2)}{k!}= 2 \left( 1- \frac{1}{\log(2)} \right) + \frac{2^z}{\log(2)} - z.
	\end{align}
	We label the leaves of $\tau_n$ from $1$ to $n$ in a uniformly at random selected way.
	\item Determine the sign of the root of $\tau_n$ according to a fair coin flip. We propagate the signs of the remaining vertices according to the parity of their height. That is, vertices receive the same sign of the root if and only if their height is even.
	\item Apply the bijection from Section~\ref{sec:bij} to form a cograph.
\end{enumerate}
\end{lemma}

The only properties of $\eta$ we are going to use is that it has expected value $\Ex{\eta} = 1$ and finite  variance $\sigma_\eta^2$. We are now ready to prove our main theorem in the labelled case:

\begin{proof}[Proof of Theorem~\ref{te:main} in the labelled case]
Let $k \ge 1$ be given and choose vertices $v_1, \ldots, v_k$ of $\mG_n$ uniformly and independently at random without replacement. By Proposition~\ref{pro:ad} we know that $v_i$ and $v_j$ are adjacent if and only if the lowest common ancestor of the corresponding leaves in $\tau_n$ is labelled with $\oplus$. This depends only on the sign of the root $o$ of $\tau_n$ and the parity of the height of the lowest common ancestor. 

The asymptotic behaviour of the subtree $\tau_n\langle o, v_1, \ldots, v_k\rangle$ spanned by $o$ and the marked vertices is known: Consider the subset $S_n \subset \tau_n  \langle o, v_1, \ldots, v_k \rangle$ of \emph{essential} vertices, given by the vertices $o$, $v_1$, \ldots, $v_k$ and the lowest common ancestors of any subset of these vertices.  Naturally we obtain a tree structure on $S_n$ from the tree structure on $\tau_n\langle o, v_1, \ldots, v_k\rangle$.

 The tree $\tau_n \langle o, v_1, \ldots, v_k \rangle$ is obtained from the tree $S_n$ of essential vertices by blowing up each edge $e$ into a path of some length $s_e$. It was shown in~\cite[Lem. 4.1]{2019arXiv190407135B} in a more general context that $S_n$ converges in distribution to the random proper $k$-tree $\mR_k$. In particular, $S_n$ has with high probability $2k -1$ edges. We may enumerate them in a canonical way  and let $s_1, \ldots, s_k$ denote the length of the corresponding paths in $\tau_n\langle o, v_1, \ldots, v_k\rangle$. It was also shown in the cited result that the vector $\mathbf{s}_n = (s_i)_{1 \le i \le 2k-1}$ admits a scaling limit
 \begin{align}
 	\label{eq:convergence}
 	\frac{\sqrt{\Pr{\eta = 0}} \sigma_\eta}{\sqrt{n}} \mathbf{s}_n \convdis \mathbf{s}
 \end{align}
 with the distribution of $\mathbf{s}$ having density
 \begin{align}
 	\label{eq:density}
 	3 \cdot 5 \cdots (2k-3) \left(\sum_{i=1}^{2k-1} x_i\right) \exp\left(- \frac{1}{2} \left(\sum_{i=1}^{2k-1} x_i\right)^2 \right), \qquad (x_i)_{1 \le i \le 2k-1} \in \ndR_{>0}^{2k-1}.
 \end{align}
 
 For each integer $s$ we let $\mathrm{par}(s)$ denote the parity (even or odd) of its length. The  corresponding local limit theorem~\cite[Lem. 4.2]{2019arXiv190407135B} for $\frac{\sqrt{\Pr{\eta = 0}} \sigma_\eta}{\sqrt{n}}\mathbf{s}_n$ entails
  that the vector $(\mathrm{par}(s_i))_{1 \le i \le 2k-1}$ converges in distribution to $2k-1$ independent fair coin flips. Consequently, the parity of the heights $(\mathrm{par}(\he_{\tau_n}(x))_{x \in S_n \setminus \{o\}})$ converges to a vector of $2k-2$ independent fair coin flips. By the definition of $\mH_k^{1/2}$, this entails
  \[
  	\mG_n[k]' \convdis \mH_k^{1/2}.
  \]
  As this holds for all $k \ge 1$,  $\mG_n \convp W_{1/2}$ follows by Condition~\eqref{eq:definite}.
 \end{proof}

  We remark that \cite{MR1207226} proved a central limit theorem as in~\eqref{eq:convergence} and also a corresponding local limit theorem  for critical Galton--Watson trees (with a finite variance branching mechanism) conditioned  on having $n$ \emph{vertices}. The results from~\cite[Lem 4.1, Lem. 4.2]{2019arXiv190407135B} that we applied here to $\tau_n$ (a Galton--Watson tree with $n$ leaves) are extensions of his work.

\section{Random unlabelled cographs}
\label{sec:unlabelled}

In this section we prove Theorem~\ref{te:main} for the unlabelled case. In principle we  pursue a similar strategy as for the labelled case, but the dealing with the symmetries of cographs requires us to use  methods for P\'olya structures~\cite{MR3773800} and unlabelled enriched trees~\cite{StEJC2018},  and to establish another extension of the skeleton decomposition by~\cite{MR1207226}.

\subsection{Enumerative preliminaries}
We let $\cA$ denote the class of unlabelled unordered rooted trees where each internal vertex has outdegree at least $2$. By the discussion in Section~\ref{sec:bij}, for any $n \ge 3$ there is a $1:2$ correspondence between trees from $\cA$ with $n$ leaves and unlabelled cographs with $n$ vertices. We let $\cA(z)$ denote the ordinary generating series of the class $\cA$ where $z$ counts the number of leaves. Any element from $\cA$ is either a tree consisting of a single vertex or a multiset of at least $2$ branches from $\cA$ dangling from a root vertex. This yields
\begin{align}
	\label{eq:yay}
	\cA(z)	& = z + \exp\left( \sum_{i\ge 1} \cA(z^i) /i \right) -1 - \cA(z).
\end{align}
We may rewrite this as $\cA(z) = E(z, \cA(z))$ for
\begin{align}
	\label{eq:alsoyay}
	E(z,y) = z + \exp(y) \exp\left(\sum_{i \ge 2} \cA(z^i)/i\right) -1 -y.
\end{align}
The following asymptotic is well known, see for example~\cite{MR2154567}. We are going to use some of the intermediate steps in the proof later on.
\begin{proposition}
	\label{pro:analytic}
	It holds as $n \to \infty$
	\begin{align}
		[z^n] \cA(z) \sim c_\cA n^{-3/2} \rho^{-n},
	\end{align}
	with $0< \rho <1$ and
	\begin{align}
	\label{eq:expr}
	c_\cA = \sqrt{\frac{\rho E_z(\rho, \cA(\rho)) }{2 \pi E_{yy}(\rho, \cA(\rho)}} \qquad \text{and} \qquad E_y(\rho, \cA(\rho)) = 1.
	\end{align}
\end{proposition}
Here $E_z$, $E_y$, and $E_{yy}$ denote partial derivatives. 

\begin{proof}[Proof of Proposition~\ref{pro:analytic}]	
	Equation~\eqref{eq:yay} entails that $\cA(z) \ge z$ and consequently 
	\begin{align}
		\label{eq:lowerbound}
		\cA(z) \ge \exp\left( \sum_{i\ge 1} \cA(z^i) /i \right) -1 - \cA(z) \ge \exp\left( \sum_{i\ge 1} z^i /i \right) - 1 - \cA(z).
	\end{align} Hence it is not possible that $\cA(t)< \infty$ for some $t>1$, because substituting $z=t$ would mean that the left hand side of this inequality is finite but the right hand side isn't. In other words, $\cA(z)$ has radius of convergence $\rho \le 1$. Equation~\eqref{eq:yay} also entails that $\cA(z) \ge c \cA(z)^2$ for some constant $c>0$. Hence it is not possible that $\cA(\rho) = \infty$, because then we could find a value $z = \rho - \epsilon$ for some $\epsilon>0$ for which the left hand side of this inequality is bigger than the right hand side. Having verified that $\cA(\rho)<\infty$, Inequality~\eqref{eq:lowerbound} also entails that $\rho<1$, since the sum inside of the exponential function would be infinite for $z=\rho$ otherwise. Moreover, $\rho=0$ is not possible, as the coefficient $[z^n]\cA(z)$ is bounded by the number of plane trees with $n$ leaves, and their generating series is known to have positive radius of convergence. Summing up, we have shown that
	\begin{align}
		\label{eq:analytic1}
		0 < \rho < 1 \qquad \text{and} \qquad \cA(\rho) < \infty.
	\end{align}
	Note that $\rho<1$ entails that $\sum_{i \ge 1} \cA( (\rho + \epsilon)^i)/i < \infty$ for $\epsilon>0$ small enough. Hence
	\begin{align}
		\label{eq:analytic2}
		E(\rho + \epsilon, \cA(\rho + \epsilon)) < \infty.
	\end{align}
	Inequalities~\eqref{eq:analytic1} and \eqref{eq:analytic2} allow us to apply standard enumerative results, see \cite[Thm. 28]{MR2240769}, yielding
	\begin{align}
		[z^n] \cA(z) \sim c_\cA n^{-3/2} \rho^{-n}
	\end{align}
	for $c_\cA$ being given by the expression in~\eqref{eq:expr}. Note that by Pringsheim's theorem the function $\cA(z)$ may not be analytically continued to a neighbourhood of $\rho$. Hence by the implicit function theorem it must hold that the function $H(z,y) := y - E(z,y)$ satisfies
	\begin{align}
		0 = H_y(\rho, \cA(\rho)) = 1 - E_y(\rho, \cA(\rho)).
	\end{align}
	This completes the verification of Equation~\eqref{eq:expr}.
\end{proof}

\subsection{An enriched tree sampling procedure}

Proposition~\ref{pro:analytic} allows us to define a  random vector $(\xi, \zeta) \in \ndN_0 \times \ndN_0$ with generating series
\begin{align}
\label{eq:gen}
\Ex{z^\xi w^\zeta} 	
&= \frac{1}{\cA(\rho)} \left(\rho  +  \exp\left(\cA(\rho) z + \sum_{i \ge 2} \cA( \rho^i w^i)/i\right) - 1 - \cA(\rho)z \right).
\end{align}
Note that $(\xi, \zeta)$ has finite exponential moments by~\eqref{eq:analytic2}, that is for some $\epsilon>0$
\begin{align}
	\label{eq:cond1}
	\Ex{(1+ \epsilon)^\xi (1+ \epsilon)^\zeta} 	< \infty.
\end{align}
Moreover,  Equation~\eqref{eq:expr} entails 
\begin{align}
	\label{eq:cond2}
\Ex{\xi} = 1.
\end{align}
We let $\mT$ denote a  Galton--Watson tree with two types of vertices, blue and red. Red vertices are infertile. Blue vertices generate offspring according to an independent copy of $(\xi, \zeta)$, with $\xi$ corresponding to blue offspring and $\zeta$ to red offspring. Of course, $\mT$ always starts with a blue root. Note that we care about the order of children with the same type, but not about the order between children with a different type.

%For any tree $T$ we let $\ell(T)$ denote the number of leaves. Note that vertices of the second type are always leaves, whereas vertices of the first type may or may not be leaves. 

We are going to make use of a blow-up construction. Suppose that for each pair $(a,b) \in \ndN_0 \times \ndN_0$ that the vector $(\xi, \zeta)$ assumes with positive probability, we are given a random  forest $\mathsf{F}_{a,b}$ of rooted plane trees with precisely $b$ leaves in total. In order to not get confused we colour the vertices of this forest green.  Let's say $T$ is a fixed (blue,red)-coloured tree that the Galton--Watson tree $\mT$ assumes with positive probability. We transform the  tree $T$ into a (blue,green)-coloured tree $\Lambda(T)$ by performing the following three steps for each blue vertex $v$ of $T$.
\begin{enumerate}
	\item Let $a(v)$ denote the number of blue children of $v$ and $b(v)$ the number of red children of $v$. Let  $\mathsf{F}(v)$ denote an independent copy of the random forest $\mathsf{F}_{a(v),b(v)}$.
	\item Delete all red offspring of $v$.
	\item For each tree in the forest $\mathsf{F}(v)$ add an edge between its root and the vertex $v$.
\end{enumerate}
Note that these local modifications do not change the blue subtree of $T$.

The tree $\Gamma(\mT)$ is a $2$-type plane tree with vertices coloured blue and green according to their type. We let $\cA(\Gamma(\mT))$ denote the result of applying a ``forgetful functor'' that reduces this to a rooted unordered unlabelled tree. We let $\ell(\cdot)$ denote the function that sends a tree (or a forest) to its number of leaves. The following Lemma is based on hidden branching processes in P\'olya structures  discovered in~\cite{MR3773800}.

\begin{lemma}
	\label{le:sampling}
	There is a family of random finite forests $(\mF_{a,b})_{a,b}$ such that $\cA(\Lambda(\mT))$ follows the Boltzmann distribution
	\begin{align}
		\Pr{\cA(\Gamma(\mT)) = A} = \rho^{\ell(A)}/\cA(\rho), \qquad A \in \cA.
	\end{align}
\end{lemma}

The distribution of the family $(\mF_{a,b})_{a,b}$ is rather technical, but it will be described in the proof. Its existence suffices for the proof our main theorem, which is why we present Lemma~\ref{le:sampling} in this way.

What is important is the big picture, that may be summarized by the following three points: First, the uniform unlabelled cograph $\mU_n$ with $n \ge 3$ vertices may be generated by sampling a uniform $\cA$-tree $\mA_n$ with $n$ leaves, assigning signs to its vertices according to a single fair coin flip, and applying the bijection from Section~\ref{sec:bij}. Second, the tree $\mA_n$ may be generated by
 conditioning $\cA(\Lambda(\mT))$ on producing a tree with $n$ leaves, which is equivalent to conditioning $\mT$ on $\ell(\mT) =n$. That is, 
 \begin{align}
 	\label{eq:poly}
 	\mA_n \eqdist \cA(\Lambda(\mT_n))
 \end{align} 
 for $\mT_n := (\mT \mid \ell(\mT) = n)$.
 Third, $\mT$ being a sesqui-type branching tree, this opens the door to methods for branching processes and random walk.

\begin{proof}[Proof of Lemma~\ref{le:sampling}]
Boltzmann sampling methods~(\cite{MR2095975,MR2810913,MR2498128}) allow us to translate the specification~\eqref{eq:yay} \emph{mechanically} into a process for generating trees from $\cA$ at random. This results for each parameter $0<x \le \rho$ in a recursive procedure $\Gamma \cA(x)$ that samples a random tree from $\cA$ with the \emph{Boltzmann distribution}
\begin{align}
	\label{eq:fob}
	\Pr{\Gamma \cA(x) = A} = x^{\ell(A)}/\cA(x).
\end{align}
For each $k \ge 0$ we let $\cS_k$ denote the symmetric group of degree $k$. Note that $\cS_0$ has a single trivial element.  For any permutation $\sigma$ and any integer $i \ge 1$ we let $\sigma_i$ denote the number of cycles of length $i$. For example, $\sigma_1$ corresponds to the fixed points. The procedure is defined as follows.
\begin{enumerate}
		\item Start with a single root vertex $v$.
		\item Draw a random permutation $\gamma$ that assumes a value $\sigma \in \bigcup_{k \ge 0} \cS_k$  with probability
		\[
			\Pr{\gamma = \sigma} = \frac{1}{\cA(\rho)} [s_1^{\sigma_1} s_2^{\sigma_2}\cdots] \left(\rho +  \exp\left(\cA(\rho) s_1 + \sum_{i \ge 2} \cA(\rho^i) s_i /i\right) - 1 - \cA(\rho)s_1 \right).
		\]
		Here $s_1, s_2, \ldots$ denote formal variables and $[s_1^{\sigma_1} s_2^{\sigma_2} \cdots]$ means that we extract the coefficient of the monomial $\prod_{i \ge 1} s_i^{\sigma_i}$.
		\item For each $i \ge 1$ do the following. Sample $\gamma_i$ independent copies $A_{i,1}, \ldots, A_{i,\gamma_i}$ of $\Gamma \cA(x^i)$ (via recursive calls to this procedure). For each $1 \le j \le \gamma_i$ make $i$ identical copies $A_{i,j,1}, \ldots, A_{i,j,i}$ of $A_{i,j}$. For each $1 \le j \le \gamma_i$ and each $1 \le k \le i$ add an edge between $v$ and the root of $A_{i,j,k}$. 
\end{enumerate}
Note that we do nothing in step $3$ if and only if the random permutation $\gamma$ drawn in step $2$ equals the trivial permutation from $\cS_0$.
Of course, we need to justify that the recursive procedure $\Gamma \cA(x)$ terminates almost surely and samples according to the distribution~\eqref{eq:fob}. A justification  is given by~\cite[Thm. 4.2]{MR2810913} in a more general context for classes that may be recursively specified as in~\eqref{eq:yay} using operations such as sums and multiset classes.

 Proposition~\ref{pro:analytic} allows us to start the process with parameter $x=\rho$. Let us colour the root (generated in step $1$) of $\Gamma \cA(\rho)$ and each recursive  call to $\Gamma \cA(\rho)$ blue. The vertex generated in the first step of any call to $\Gamma \cA(\rho^i)$ for $i \ge 2$ gets coloured green.  Note that the result is a tree where any green vertex has only green descendants. That is, it consists of a tree of blue vertices, where each (blue) vertex $v$ is connected via single edges to a forest $F(v)$ of green trees.  Compare with \cite[Fig. 1]{MR3773800}. 

We let $\mathbf{f}$ denote the number of blue offspring of the root of $\Gamma \cA(\rho)$, and $\mathbf{F}$ the green forest attached to the root. We may reformulate the procedure $\Gamma \cA(\rho)$ as follows. Start with a blue vertex that is marked unvisited. In each step, we select an arbitrary unvisited blue vertex which is then marked as visited and receives blue unvisited offspring and a green forest according to an independent copy of $(\mathbf{f}, \mathbf{F})$. These steps are repeated until the process dies out. 

The pair $(\mathbf{f}, \ell(\mathbf{F}))$ is distributed like the vector $(\xi, \zeta)$ described in Equation~\eqref{eq:gen}. Defining  for all pairs $(a,b) \in \ndN_0 \times \ndN_0$ with $\Pr{(\xi, \zeta) = (a,b)}>0$ the conditioned forest $\mF_{a,b} = (\mathbf{F} \mid (\mathbf{f}, \ell(\mathbf{F})))$ to be independent from $(\xi, \zeta)$, it holds that
\begin{align}
	(\mathbf{f}, \mathbf{F}) \eqdist \mF_{\xi, \zeta}.
\end{align} 
Consequently,
\begin{align}
\cA(\Lambda(\mT)) \eqdist \Gamma \cA(\rho)
\end{align}
follows the Boltzmann distribution from Equation~\eqref{eq:fob} for $x = \rho$.
\end{proof}

\subsection{Spooky scary skeletons}

The skeleton decomposition by
\cite{MR1207226} is a local limit theorem for the subtree spanned by any fixed number of random vertices in a critical Galton--Watson tree (subject to a finite variance constraint on the branching mechanism) conditioned on having a large number of vertices.

In this section we are going to extend his result (following closely his arguments) to conditioned sesqui-type trees ($2$-type Galton--Watson trees where only one type is fertile) and obtain additional information on the vicinity of the joints. A similar extension was recently given in~\cite{2019arXiv190407135B} for mono-type Galton--Watson trees conditioned on having a large number of vertices with outdegree in some fixed set. The limits we are going to establish in this section also generalize earlier results by~\cite[Thm. 27, Thm. 24]{StEJC2018}, that describe the $o(\sqrt{n})$-neighbourhood of the vicinity of a fixed or random root in sesqui-type trees and related random unlabelled graphs and trees, including P\'olya trees.

Throughout this section, we let $\mT$ denote a general sesqui-type tree with a non-degenerate offspring  distribution $\bm{\eta} := (\xi, \zeta)$  satisfying Conditions~\eqref{eq:cond1} and \eqref{eq:cond2}. We additionally make the aperiodicity assumption
\begin{align}
	\label{eq:aperiodic}
	\gcd\{k \ge 0 \mid \Pr{\xi= k}\} = 1.
\end{align}
 We keep the notation that vertices of the first type are coloured blue, and vertices of the infertile second type are coloured green. We let $\mT_n$ denote the result of conditioning $\mT$  on having $n$ leaves. Of course, we will later go back to applying this result to the specific offspring distribution described in Equation~\eqref{eq:gen}. However, stating our extension in this more general form comes at no extra cost and may prove useful in other contexts, as sesqui-type trees have received some attention in recent literature such as~\cite{2017arXiv170600283J}, and are connected to the behaviour of a variety of random objects, such as Achlioptas processes~\cite{2017arXiv170408714R} and random unlabelled graphs~\cite{StEJC2018}.

\subsubsection{Preliminary properties of $\mT$}

The blue subtree of $\mT$ follows the law of a mono-type Galton--Watson tree with offspring law $\xi$. Assumption~\eqref{eq:cond2} entails that it is critical, and hence almost surely finite. It follows that $\mT$ is almost surely finite.  Moreover, each leaf of the blue subtree has a constant independent non-zero chance of being a leaf of $\mT$. As the number of leaves of the blue subtree is (like for every critical mono-type Galton--Watson tree) heavy-tailed, it follows that the number $\mZ$ of leaves in $\mT$ is heavy-tailed as well. Hence the  probability generating function $\cZ(z) := \Ex{z^\mZ}$ has radius of convergence $\rho_\cZ = 1$. Letting $f(z,w) = \Ex{z^\xi w^\zeta}$ denote the bivariate probability generating function of the offspring distribution $\bm{\eta}$, it holds that
\begin{align}
\cZ(z) = \Pr{ \bm{\eta} = (0,0) }(z-1) + f(\cZ(z),z).
\end{align}
We may rewrite this as
\begin{align}
\cZ(z) = F(z, \cZ(z)) \qquad \text{with} \qquad  F(z,y) = \Pr{ \bm{\eta} = (0,0) }(z-1) + f(y,z).
\end{align}
Assumption \eqref{eq:cond1} ensures that for some $\epsilon>0$ \begin{align}
F(\rho_\cZ + \epsilon, \cZ(\rho_\cZ) + \epsilon) = F(1 + \epsilon, 1+\epsilon) < \infty
\end{align} 
Using the aperiodicity assumption~\eqref{eq:aperiodic}, it follows by \cite[Thm. 28]{MR2240769}
\begin{align}
\label{eq:stb}
[z^n] \cZ(z) \sim c_\cZ n^{-3/2}
\end{align}
with
\begin{align}
\label{eq:ct}
c_\cZ = \sqrt{\frac{ F_z(1, 1) }{2 \pi F_{yy}(1, 1)}} = \sqrt{\frac{\Pr{\bm{\eta} = (0,0)} + \Ex{\zeta}}{2 \pi \Va{\xi}}}.
\end{align}
Equation~\eqref{eq:stb} implies that $\mZ$ lies in the domain of attraction of the positive $1/2$ stable law. For each $n \ge 0$ we let $S_n$ denote the sum of $n$ independent copies of $\mZ$. Setting 
\begin{align}
\label{eq:sigma}
\sigma := \sqrt{ \frac{\Va{\xi}}{\Pr{\bm{\eta}=(0,0)} + \Ex{\zeta}} },
\end{align}
it follows by~\cite[Sec. 50]{MR0062975} that
\begin{align}
\label{eq:loclim}
\lim_{n\to \infty} \sup_{r \ge 0} \left|n^2 \P\left( S_n = r\right) - \sigma^2 g(\sigma^2 r / n^2) \right| = 0.
\end{align}
Here $g$ denotes the positive stable $1/2$-density given by
\begin{align}
\label{eq:stabledensity}
g(x) = (2 \pi)^{-1/2} x^{-3/2} \exp\left(-\frac{1}{2x}\right), \qquad x > 0.
\end{align}

\begin{comment}

\subsubsection{The colour of a random leaf}

We let $\bm{\eta}_i = (\xi_i, \zeta_i)$, $i \ge 1$, denote independent copies of~$\bm{\eta}$. We may generate $\mT$ from this list by letting the root receive offspring according $\bm{\eta}_1$. The root is then marked as visited, and its blue offspring as unvisited. We then proceed in each step with the lexicographically first unvisited blue vertex, mark it as visited, and assign to it offspring according to the next unused list entry of $(\bm{\eta}_i)_{i \ge 1}$. 

Thus, the number $B$ of blue vertices in $\mT$ is given by the smallest integer $B$ for which
\begin{align}
	\sum_{i=1}^B (\xi_i -1) = -1.
\end{align}
The number $L_\mathrm{b}$ of blue leaves is consequently given by
\begin{align}
	L_\mathrm{b} = \sum_{i=1}^B \one_{(\xi_i,\zeta_i) = (0,0)}.
\end{align}
The number $G$ of green vertices and the total number $V$ of vertices in $\mT$ are then given by
\begin{align}
	 V = \sum_{i=1}^B (\zeta_i +1) \qquad \text{and} \qquad G = \sum_{i=1}^B \zeta_i.
\end{align}
Any green vertex is a leaf, hence the total number $L$ of leaves is given by
\begin{align}
	L = L_{\mathrm{b}} + G = \sum_{i=1}^B \one_{(\xi_i,\zeta_i) = (0,0)} + \sum_{i=1}^B \zeta_i.
\end{align}

\end{comment}

\subsubsection{The limit object}

We define the biased versions $\bm{\eta}^\bullet$, $\bm{\eta}^*$,  $\bm{\eta}^\circ$ of the offspring distribution $\bm{\eta}$ with distribution given by
\begin{align}
\Pr{\bm{\eta}^\bullet = (a,b)} &= a \Pr{\bm{\eta} = (a,b)}, \\
\Pr{\bm{\eta}^* = (a,b)} &= a(a-1)\Pr{\bm{\eta} = (a,b)} / \Va{\xi}, \\
\Pr{\bm{\eta}^\circ = (a,b)} &= b\Pr{\bm{\eta} = (a,b)} / \Ex{\zeta}.
\end{align}
Conditions~\eqref{eq:cond1} and \eqref{eq:cond2} ensure that these are well-defined probability distributions.

Let us  provide some intuition for what we are about to do. If we choose a blue leaf of $\mT_n$ uniformly at random, then, by the waiting time paradox, we expected its parent to asymptotically behave like $\bm{\eta}^\bullet$ and not $\bm{\eta}$. The reason being that vertices with many blue vertices are more likely candidates for being the parent. Likewise, if we choose a green leaf of $\mT_n$ uniformly at random, then its parent should behave like $\bm{\eta}^\circ$ as $n$ tends to infinity. Furthermore, we expect $\bm{\eta}^*$ to quantify the asymptotic behaviour of the lowest common ancestor of two randomly selected leaves of $\mT_n$, because such an ancestor would have two distinguished blue offspring vertices leading to the selected leaves. If we select a leaf in $\mT_n$ uniformly at random, our intuition is that it is  blue with a probability tending to $\Pr{\bm{\eta}=(0,0)} / (\Pr{\bm{\eta}=(0,0)} + \Ex{\zeta})$. Moreover, we expect the tree $\mT_n$ to lie in the universality class of the Brownian tree, implying  that the genealogical structure of $k$ uniformly selected leaves and the root is  asymptotically quantified by the random proper $k$-tree $\mR_k$ defined in Section~\ref{sec:limobj}.

 Guided by our intuition, we define a limit object and check convergence later. For any integers $k, t \ge 1$ we construct a random $2$-type tree $\cT^{k,t}$ with $k$ distinguished leaves (with labels from $1$ to $k$) as follows in three steps.
\begin{enumerate}
	\item \emph{Stretch the skeleton.}
 We select a vector $\bm{s} = (s_i)_i \in \mathbb{R}_{>0}^{2k-1}$ at random with density 
 \begin{align}
 \label{eq:densityh}
 h(\bm{x}) =  \left(\prod_{i=1}^{k-1}(2i-1) \right) \left(\sum_{i=1}^{2k-1} x_i\right) \exp\left(- \frac{1}{2} \left(\sum_{i=1}^{2k-1} x_i\right)^2 \right).
 \end{align}
 Next, for each $1 \le i \le 2k-1$, we replace the $i$th edge of $\mR_k$ by a path of length~$2t + 1$ and label the middle edge with $s_i$.
 \item \emph{Paint the blow-up}  We colour all non-leaves of this blow-up blue. Each leaf of the blow-up is coloured blue with probability $\Pr{\bm{\eta}=(0,0)} / (\Pr{\bm{\eta}=(0,0)} + \Ex{\zeta})$ and otherwise green. Let us denote the resulting coloured blow-up by $\mB_k$.
 	\item \emph{Add local growth.} Let $v$ iterate over all vertices of $\mB_k$ that are not leaves. There are $3$ cases:
 	\begin{enumerate}
 	\item  If $v$ has a single blue offspring $u$, then $v$ receives additional offspring according to an independent copy of $\bm{\eta}^\bullet-(1,0)$. Each additional blue child  becomes the root of an independent copy of $\mT$. The location of $u$ among the additional blue childs of $v$ is selected uniformly at random.
 	\item  If $v$ has a single green child $u$, then $v$ receives additional offspring according to an independent copy of $\bm{\eta}^\circ-(0,1)$. The location of $u$ among the additional green children of $v$ is selected uniformly at random.  Each additional blue child becomes the root of an independent copy of $\mT$.
 	\item If $v$ has two blue children,  then it receives additional offspring according to an independent copy of $\bm{\eta}^*-(2,0)$. The location of the pre-existing blue children  $u_1$ and $u_2$  among the additional blue children is selected uniformly at random in a way that preserves the relative order between $u_1$ and $u_2$. 
 	\end{enumerate}
 	After $v$ iterated over all vertices of $\mB_k$ we are left with a $2$-type tree denoted by $\cT^{k,t}$. We let $e_1, \ldots, e_{2k-1}$ denote the edges to which we assigned labels $s_1, \ldots, s_{2k-1}$.
\end{enumerate}

\subsubsection{The skeleton decomposition}

Let us first describe the distribution of $\cT^{k,t}$ a bit more explicitly. Let $T$  be a $2$-type tree (with the second type being infertile) with $k$ marked leaves that are labelled from $1$ to $k$. The \emph{essential vertices} of $T$ consist of the root of $T$, the marked leaves, and the lowest common ancestors of any pair of marked leaves. Let $R(T)$ denote the tree induced \emph{on} the collection of essential vertices.  The tree $S(\cT^{k,t})$ obtained by contracting each of the $2k-1$ labelled edges has the property, that any two adjacent vertices in $R(\cT^{k,t})$ have distance precisely $2t$ in $S(\cT^{k,t})$. 
\begin{lemma}
	\label{le:lem1}
	Suppose that $T$ has the property that $R(T)$ is  a proper $k$-tree and any adjacent vertices in $R(T)$ have distance $2t$ in $T$. Then 
	\begin{align}
		\label{eq:seq}
		\Pr{S(\cT^{k,t}) = T} = \left(  \prod_{i=1}^{k-1}(2i-1) \right)^{-1} \Pr{\mT= T} \frac{\sigma^{2k}}{\Va{\xi}^{2k-1}}.
	\end{align}
\end{lemma}
\begin{proof}	
	The tree $\mR_k$ used in the first step of the construction of $\cT^{k,t}$ is equal to $R(T)$  with probability
	\begin{align}
		\label{eq:part1}
		\left(  2^{k-1} \prod_{i=1}^{k-1}(2i-1) \right)^{-1}.
	\end{align}
	Conditional on this event, there are unique choices for each colouring in step 2 and each outdegree in step $3$ so that $S(\cT^{k,t}) = T$.
		
	Let $(a,b) \in \ndN \times \ndN_0$ and $1 \le a_1 \le a$ be given. 
	If a vertex receives offspring according to $\bm{\eta}^\bullet$ and we distinguish a uniformly selected blue offspring, then the probability to produce offspring $(a,b)$ with precisely $a_1$th blue vertex being distinguished is equal to  $\Pr{\eta = (a,b)}$.
	
 	Suppose that additionally $a \ge 2$ and $1 \le a_1 < a_2 \le 2$ are given. If a vertex receives offspring according to $\bm{\eta}^*$ and we distinguish a uniformly selected $2$-element subset of blue children, then probability for producing offspring $(a,b)$ with precisely the $a_1$th and $a_2$th vertices being marked is equal to $\Pr{\eta= (a,b)} 2 / \Va{\xi}$. There are precisely $k-1$ vertices in the construction of  $\cT^{k,t}$ that receive offspring according to an independent copy of $\bm{\eta}^*$, hence each contributes an additional factor $ 2 / \Va{\xi}$.  
 	
 	Likewise, each of the marked leaves contributes an additional factor $1/(\Pr{\bm{\eta}=(0,0)} + \Ex{\zeta}).$ Hence, taking the product over all vertices in $T$, we arrive at the conditional probability
 	\begin{align}
 		\label{eq:part2}
 		\Pr{\mT= T} \left( \frac{2}{\Va{\xi}} \right)^{k-1} \frac{1}{(\Pr{\bm{\eta}=(0,0)} + \Ex{\zeta})^k}.
 	\end{align}
 	Taking the product of the probabilities in \eqref{eq:part1} and \eqref{eq:part2}, we arrive at the formula~\eqref{eq:seq}.
\end{proof}

The idea we want to express is that if $t=t_n = o(\sqrt{n})$ then $\mT_n$ looks asymptotically like the result of replacing for each $1 \le i \le 2k-1$ the edge $e_i$ in $\cT^{k,t}$ by some unspecified  bi-pointed tree $T_{i,n}$ whose two root vertices have distance proportional to $s_i \sqrt{n}$ (up to a constant factor that does not depend on $i$ or $n$).

To this end, let $\bm{\ell} = (\ell_1, \ldots, \ell_{2k-1})$ be a vector of positive integers and set $|\bm{\ell}| = \sum_{i =1}^{2k-1} \ell_i$.
Suppose that $T$ has the property that $R(T)$ is  a proper $k$-tree and any adjacent vertices in $R(T)$ have distance $2t$ in $T$. We let $\ell(T)$ denote the number of leaves of $T$. Let $\cE_n(T,\bm{\ell})$ denote the set of finite $2$-type trees  with $n$ leaves that have $k$ marked leaves that are labelled from $1$ to $k$ and may be obtained from $T$ by cutting open for each $1 \le i \le 2k-1$ the middle of the path corresponding to the $i$th edge of $R(T)$ and inserting an arbitrary tree $T_i$ such that the distance increases from $2t$ to $2t + \ell_i$. That is, we insert a tree with a root and a marked leaf (of height $\ell_i$) and identify one of the ends with the root and the other with the marked leaf.

\begin{lemma}
	\label{le:lem2}
	Let $\bm{v}$ denote  $k$ leaves of $\mT$, labelled from $1$ to $k$, that we select uniformly with replacement. Let $(\eta_1^\bullet(j), \eta_2^\bullet(j))_{j \ge 1}$ denote independent copies of $\bm{\eta}^\bullet$.  Then
	\begin{align}
		\label{eq:theshit}
		\Pr{(\mT, \bm{v}) \in \cE_n(T, \bm{\ell})} = n^{-k} \Pr{\mT=T}  \Prb{S_{\sum_{j=1}^{|\bm{\ell}|}(\eta_1^\bullet(j) -1)} = n - \ell(T) - \sum_{j=1}^{|\bm{\ell}|}\eta_2^\bullet(j)	}.
	\end{align}
\end{lemma}
\begin{proof}
	We let $\cE_n^*(T,\bm{\ell}) \subset \cE_n(T,\bm{\ell})$ denote the subset of values that $(\mT, \bm{v})$ attains with positive probability.
	An element $X$ of $\cE_n^*(T,\bm{\ell})$ is fully characterized by the corresponding leaf-marked trees $(T_i)_{1 \le i \le 2k-1}$. The probability for $(\mT, \bm{v})$ to equal $X$ is given by
	\begin{align}
		\label{eq:expr1}
		n^{-k} \Pr{\mT=T} \prod_{i=1}^{2k-1} \prod_{v \in T_i^*} \Pr{ \bm{\eta} = d_{T_i}^+(v)}.
	\end{align}
	Here $T_i^*$ denotes the vertex set of the tree $T_i$ without the marked leaf. The outdegree $d_{T_i}^+(v)$ consists of number of blue offspring and green offspring of the vertex $v$ in $T_i$. Summing over all~$X \in \cE_n^*(T,\bm{\ell})$ means that the family $(T_i)_{1 \le i \le 2k-1}$ ranges over leaf-marked $2$-type trees subject to the constraint that the marked leaf in $T_i$ has colour blue and height $\ell_i$, and that the number of non-marked leaves in the $T_i$ sum up to $n- \ell(T)$.
	
	Each tree $T_i$ has a \emph{spine}, given by the path connecting the root with the marked tree. The total number of non-marked leaves of $T_i$ is given by the green offspring along the spine plus the leaves of the trees attached to the blue offspring along the spine.
	Summing over all  possible offspring-marked outdegrees along the spines, we obtain that Expression~\eqref{eq:expr1} equals
	\begin{align*}
		n^{-k} \Pr{\mT=T}  \sum_{(a_1,b_1), \ldots, (a_{|\bm{\ell}|}, b_{|\bm{\ell}|}) \in \ndN_0 \times \ndN_0}  \Prb{S_{\sum_{j=1}^{|\bm{\ell}|}(a_j -1)} = n - \ell(T) - \sum_{j=1}^{|\bm{\ell}|}b_j  } \prod_{j=1}^{|\bm{\ell}|} \Pr{\bm{\eta} = (a_j,b_j)} a_j.
	\end{align*}
	This expression may be viewed as the expectation of a conditional expectation. Hence it equals
	\begin{align}
		n^{-k} \Pr{\mT=T}  \Prb{S_{\sum_{j=1}^{|\bm{\ell}|}(\eta_1^\bullet(j) -1)} = n - \ell(T) - \sum_{j=1}^{|\bm{\ell}|}\eta_2^\bullet(j)  }.
	\end{align}
\end{proof}

We would like to compute the asymptotic behaviour of the probability in Equation~\eqref{eq:theshit}.

\begin{lemma}
	\label{le:lem3}
	It holds uniformly as $n \to \infty$, $|\bm{\ell}| = \Theta(\sqrt{n})$, and $\ell(T)=o(n)$ that
	\begin{align}
	\label{eq:sumto1}
	\Prb{S_{\sum_{j=1}^{|\bm{\ell}|}(\eta_1^\bullet(j) -1)} = n - \ell(T) - \sum_{j=1}^{|\bm{\ell}|}\eta_2^\bullet(j)	} \sim  \frac{\Va{\xi}}{\sqrt{2\pi \sigma^2}} |\bm{\ell}| n^{-3/2} \exp\left( -\frac{ (\Va{\xi} |\bm{\ell}|)^2}{2 \sigma^2n} \right).
	\end{align}
\end{lemma}
\begin{proof}
	For ease of notation, let us set $L_1 = \sum_{j=1}^{|\bm{\ell}|}(\eta_1^\bullet(j) -1)$, $L_2 = \sum_{j=1}^{|\bm{\ell}|}\eta_2^\bullet(j)$, and $\ell=|\bm{\ell}|$. Note that
	\[
	\Ex{ \bm{\eta}^\bullet - (1,0)} = (\Va{\xi}, \Ex{\xi \zeta}).
	\]
	 Since $|\bm{\ell}| = \Theta(\sqrt{n})$ and since $\bm{\eta}$ has finite exponential moments, it follows that there is a sequence $\epsilon_n \to 0$ and constants $0< \delta < 1$ and $C,c>0$ such that
	\[
		\Prb{ L_1 \notin (1 \pm \epsilon_n) \ell \Va{\xi} \text{ or } L_2 \notin (1 \pm \epsilon_n) \ell \Ex{\xi \zeta}} \le C\exp(-cn^{\delta}).
	\]
	It follows by conditioning on $(L_1, L_2)$ and applying~\eqref{eq:loclim} that 
		\begin{align}
	\label{eq:sumto2}
	\Prb{S_{L_1} = n - \ell(T) - L_2} \sim \frac{1}{(\Va{\xi}\ell)^2} \sigma^2 g(\sigma^2 n / (\Va{\xi}\ell)^2).
	\end{align}
	By Equation~\eqref{eq:stabledensity} this yields Equation~\eqref{eq:sumto1}.
\end{proof}

Having Lemmas~\ref{le:lem1}, \ref{le:lem2}, and \ref{le:lem3} at hand, and knowing the probability for $\mT$ to have $n$ leaves from Equations~\eqref{eq:stb} and \eqref{eq:ct}, we obtain by an elementary calculation the following expression involving the density from Equation~\eqref{eq:densityh}:
\begin{theorem}
	\label{te:skeldecomp}
	It holds uniformly as $n \to \infty$, $|\bm{\ell}| = \Theta(\sqrt{n})$, and $\ell(T)=o(n)$ that
	\begin{align}
	\Pr{(\mT_n, \bm{v}) \in \cE_n(T, \bm{\ell})} %&= \Pr{S(\cT^{k,t}) = T} \left(  \prod_{i=1}^{k-1}(2i-1) \right) n^{-k} \left(\frac{\Va{\xi}}{\sigma}\right)^{2k}  |\bm{\ell}| \exp\left( -\frac{ (\Va{\xi} |\bm{\ell}|)^2}{2 \sigma^2n} \right). \\
	&\sim  \Pr{S(\cT^{k,t(T)}) = T} \left( \frac{\Va{\xi}}{\sigma \sqrt{n}}\right)^{2k-1} h\left(\frac{\Va{\xi}}{\sigma \sqrt{n}} \bm{\ell} \right).
	\end{align}
\end{theorem}
Here we let $t(T)$ denote the cutting distance corresponding to $T$, obtained by taking half of the distance in $T$ between any pair of essential vertices that are neighbours in $R(T)$. 
Theorem~\ref{te:skeldecomp} shows that $R(\mT_n, \bm{v})$ asymptotically behaves like a uniformly selected proper $k$-tree, that is
\begin{align}
\label{eq:convR}
R(\mT_n, \bm{v}) \convdis \mR_k.
\end{align}
Theorem~\ref{te:skeldecomp} also entails that the (with high probability $2k-1$ dimensional) vector of distances $\bm{s}(\mT_n, \bm{v})$ corresponding to the distances between essential vertices in $\mT_n$ that are neighbours in $R(\mT_n, \bm{v})$ admits the distribution with density $h(\bm{x})$ given in \eqref{eq:densityh} as scaling limit, specifically
\begin{align}
\label{eq:inya}
 \frac{\Va{\xi}}{\sigma \sqrt{n}} \bm{s}(\mT_n, \bm{v}) \convdis \bm{s}.
\end{align}
It is even a local limit theorem, as it entails that uniformly for $\bm{\ell} = \Theta(\sqrt{n})$
\begin{align}
	\Pr{ \bm{s}(\mT_n,\bm{v}) = \bm{\ell}} = \left( \frac{\Va{\xi}}{\sigma \sqrt{n}}\right)^{2k-1} h\left(\frac{\Va{\xi}}{\sigma \sqrt{n}} \bm{\ell} \right).
\end{align}
In particular, the parities $\mathrm{par}(\bm{s}(\mT_n, \bm{v}))$ having (with high probability) values in $\{\text{even}, \text{odd}\}^{2k-1}$ satisfy
\begin{align}
	\label{eq:convpar}
	\mathrm{par}(\bm{s}(\mT_n, \bm{v})) \convdis (X_i)_{1 \le i \le 2k-1},
\end{align}
with $(X_i)_{1 \le i \le 2k-1}$ denoting a family of fair independent coin flips.

\begin{lemma}
	\label{le:lem}
	If $t_n = o(\sqrt{n})$ is a sequence of integers tending to infinity, then the number of leaves in $S(\cT^{k, t_n})$ is $o_p(n)$.
\end{lemma}
\begin{proof}
	We use the notation from the proof of Lemma~\ref{le:lem3} and let $(\eta_1^*(j), \eta_2^*(j))_{j \ge 1}$ denote independent copies of $\eta^*$. It follows from the construction of $\cT^{k, t_n}$ that the number of leaves in $S(\cT^{k, t_n})$ may stochastically be bounded by
	\[
		S_{\sum_{j=1}^{(2k-1)t_n}\eta_1^\bullet(j) + \sum_{j=1}^{k-1} \eta_1^*(j) } + \sum_{j=1}^{(2k-1)t_n} \eta_2^\bullet(j) + \sum_{j=1}^{k-1} \eta_2^*(j).
	\]
	The second and third summand are $o_p(n)$ by Markov's inequality and $t_n = o(n)$. Concentration inequalities for sums of random variables with finite exponential moments ensure that the number of summands in the first term concentrates around $t_n$. As $t_n^2 = o(n)$, it follows by bounds for sums of heavy tailed random variables given in~\cite{cline1989large} that this bound is $o_p(n)$.	
\end{proof}

For $t_n = o(\sqrt{n})$ we may define $S(\mT_n, \bm{v}_n)$ analogously to $S(\cT^{k,t_n})$ by contracting the middle segments (and attached trees) on paths between essential vertices that are neighbours in $R(\mT_n, \bm{v}_n)$. That is, we cut the path at two points, each having distance $t_n$ from its closest essential vertex, throw away the middle segment, and identify the two ends where we cut. Lemma~\ref{le:lem} and Theorem~\ref{te:skeldecomp} together imply that
\begin{align}
	\label{eq:doh}
	d_{\textsc{TV}}( S(\cT^{k,t_n}), S(\mT_n, \bm{v})) \to 0.
\end{align}

\subsection{Applying the skeleton decomposition to random unlabelled cographs}

Having done all preparations, we may treat the unlabelled case with little effort.

\begin{proof}[Proof of Theorem~\ref{te:main} in the unlabelled case]
The uniform unlabelled cograph $\mU_n$ with $n \ge 3$ vertices may be generated from the random P\'olya tree $\mA_n$ with $n$ leaves by assigning signs (either $\oplus$ or $\ominus$) to its vertices, and applying the bijection from Section~\ref{sec:bij}. The signs are determined by selecting the sign of the root according to a single fair coin flip, and letting the any other vertex have the same sign if and only if its height is even. Equation~\eqref{eq:poly} states that $\mA_n \eqdist \cA(\Lambda(\mT_n))$ may be generated by applying the $\Lambda$ operator to a specific $2$-type tree $\mT_n$ where the second type is infertile and the first has offspring distribution given in Equation~\eqref{eq:gen}. Conditions~\eqref{eq:cond1} and \eqref{eq:cond2} allow us to apply the skeleton decomposition developed in the previous section. Hence, if we select a vector $\bm{v}_n$ of $n$ leaves of $\mT_n$ uniformly at random, then Equation~\eqref{eq:convR} entails that the induced tree structure on the corresponding essential vertices converges in distribution to a uniformly selected proper $k$-tree. Equation~\eqref{eq:convpar} entails that the parity of the height of the essential non-root vertices converges to a vector of fair independent coin flips. It follows from  Theorem~\ref{te:skeldecomp} that the same holds for the tree $\cA(\Lambda(\mT_n))$. By the definition of $\mH_k^{1/2}$, it follows that
\begin{align}
\mU_n[k]' \convdis \mH_k^{1/2}.
\end{align}
This holds for all $k \ge 1$. Hence Condition~\eqref{eq:definite} implies that
\begin{align}
\mU_n \convp W_{1/2}.
\end{align}
\end{proof}

\bibliographystyle{alea3}
\bibliography{cographs}

\end{document}